\documentclass[11pt]{amsart}

\usepackage[english]{babel}

\usepackage{amsmath,amssymb,amscd}

\usepackage{graphicx}
\usepackage{subcaption}

\usepackage{hyperref}

\title{Coxeter polytopes and Benjamini--Schramm convergence}

\author{Jean Raimbault}
\address{Institut de Mathématiques de Marseille, UMR 7373, CNRS, Aix-Marseille Université}
\email{jean.raimbault@univ-amu.fr}

\DeclareFontFamily{U}{wncy}{}
\DeclareFontShape{U}{wncy}{m}{n}{<->wncyr10}{}
\DeclareSymbolFont{mcy}{U}{wncy}{m}{n}
\DeclareMathSymbol{\Sha}{\mathord}{mcy}{"58}

\newcommand{\sub}{\mathrm{Sub}}
\newcommand{\eps}{\varepsilon}

\newcommand{\pl}{\partial}
\newcommand{\bs}{\backslash}

\newcommand{\vol}{\operatorname{vol}}
\newcommand{\area}{\operatorname{Area}}

\newcommand{\interval}[4]{
  \ifthenelse{ \equal{#1}{o} } {\mathopen{]}} {\mathopen{[}}
  #2, #3
  \ifthenelse{ \equal{#4}{o} } {\mathclose{[}} {\mathclose{]}}
}

\newcommand{\PO}{\mathrm{PO}}

\newcommand{\PGL}{\mathrm{PGL}}

\newcommand{\isom}{\mathrm{Isom}}

\newcommand{\RR}{\mathbb R}

\newcommand{\HH}{\mathbb H}

\numberwithin{equation}{section}

\begin{document}

\newtheorem{theo}{Theorem}
\newtheorem{lem}[theo]{Lemma}
\newtheorem{prop}[theo]{Proposition}
\newtheorem{cor}[theo]{Corollary}
\newtheorem{question}{Question}

\begin{abstract}
  We observe that a large part of the volume of a hyperbolic polyhedron is taken by a tubular neighbourhood of its boundary, and use this to give a new proof for the finiteness of arithmetic maximal reflection groups following a recent work with M.~Fr\c aczyk and S.~Hurtado. We also study in more depth the case of polygons in the hyperbolic plane. 
\end{abstract}

\maketitle

Let $X$ be a space of constant curvature, that is either a hyperbolic space $\HH^d$, a Euclidean space $\RR^d$ or a sphere $\mathbb S^d$. An hyperplane in $X$ is a one-lower-dimensional complete totally geodesic subspace, and a polytope is a bounded (or in the case of $\HH^d$, finite-volume) region delimited by a finite number of hyperplanes. A polytope in $X$ is said to be Coxeter if the dihedral angles between its faces are each of the form $\pi/m$ for some $m \ge 2$. For Coxeter polytopes in $\RR^n$ or $\mathbb S^2$ there is a well-known, complete and very intelligible classification of Coxeter polytopes by Coxeter diagrams. On the other hand Coxeter polytopes in $\HH^d$ have a very different behaviour and are still quite mysterious. In the sequel we will thus be only concerned with $X = \HH^d$. We will study Coxeter polytopes from a metric viewpoint and establish some results about their shapes when the volume tends to infinity, especially when $d=2$. A general survey on Coxeter groups in hyperbolic space is given in \cite{Vinberg_survey}; a more recent one is \cite{Belolipetsky_survey} which focuses on arithmetic aspects. 

\medskip

Let $P$ be a Coxeter polytope in $\HH^d$, and let $\Gamma_P$ be the subgroup of $\isom(\HH^d)$ generated by reflections in the faces of $P$. This is a discrete subgroup acting on $X$ with fundamental domain $P$, by the Poincar\'e polyhedron theorem. Moreover, $P$ is identified with the $\HH^d$-orbifold given by $\Gamma_P \bs \HH^d$, by endowing each face of $P$ with the local orbifold structure given by its pointwise stabiliser (the group generated by reflexions in the maximal faces which contain it). For $R>0$ thee $R$-thin part of $P$ is given by
\[
P_{\le R} = \{ x \in P : \exists \gamma\in\Gamma_P,\, d(x, \gamma x) \le R/2 \}
\]
and it corresponds to the $R$-thin part of the orbifold. An easy exercise shows that $P_{\le R}$ is equal to the set of points in $P$ which are at distance at most $R/2$ from the boundary. The first result in this note is the following, which is essentially a consequence of the hyperbolic isoperimetric inequality as we prove in \ref{proof_thm1} below. 

\begin{theo} \label{main}
  For every $d \ge 2$ there exists a constant $C(d)$ such that for every Coxeter polytope $P$ of finite volume in $\HH^d$ we have
  \[
  \vol(P_{\le 2}) \ge C(d)\vol(P).
  \]
\end{theo}

%\begin{cor} \label{BSconv-general}
%  If $\mu$ belongs to the closure of the set of all $\mu_P$ for $P$ a finite-volume Coxeter polytope in $\HH^d$ then a postive $\mu$-proportion of subgroups of $G$ contain a reflection (in paticular they are not central in $G$). 
%\end{cor}

In a joint work with M.~Fr\c aczyk and S.~Hurtado \cite[Theorem D]{aritmargulis} it was proven that $\vol(M_{\le R}) = o(\vol M)$ uniformly for $M$ a congruence arithmetic orbifold quotient of a given symmetric space. This together with Theorem \ref{main} immediately implies the following result, which was originally proved by Nikulin and Agol--Belolipetsky--Storm--Whyte (\cite{nikulin,ABSW} respectively). In fact the inspiration for this note was provided by a recent work of Fisher--Hurtado \cite{fisher_hurtado} where they use a lower-level part of \cite{aritmargulis}\footnote{Namely the ``arithmetic Margulis lemma'', Theorem 3.1 in loc. cit. which is an essential ingredient in the proof of Theorem D in loc. cit.} to give a new proof of Nikulin and Agol--Belolipetsky--Storm--Whyte's result. 

\begin{cor}
  For any $d$ there are at most finitely many congruence reflection groups in $\PO(d, 1)$. 
\end{cor}

Note that this is slightly different from the main results stated in \cite{nikulin,ABSW}. The difference is that they prove finiteness of {\em maximal reflective} arithmetic lattices (that is, arithmetic lattices wich are also generated by reflections in the sides of a Coxeter polytope, and are maximal with respect to this property), and these may be non-congruence (see \cite{lakeland}). However a short argument due to Vinberg, and used by Agol--Belolipetsky--Storm--Whyte, shows that an arithmetic maximal reflection group has uniformly bounded index in its congruence closure (see \cite[Lemma 6.2]{ABSW}); hence we can also deduce their original result from our proof.

\begin{cor}
  For any $d$ there are at most finitely many arithmetic maximal reflection groups in $\PO(d, 1)$. 
\end{cor}

It is also well-known that cofinite reflection groups cannot exist in large dimensions \cite{Vinberg_35,Khovanskii}, so we may as well say that there are only finitely many congruence (or maximal arithmetic) hyperbolic reflection groups. 

\medskip

In the case of polygons in $\HH^2$ we can say more. Let $G = \isom(\HH^d) = \mathrm{PO}(d, 1)$, let $\mu_P$ be the $G$-invariant measure on the Chabauty space $\sub_G$ of $G$ supported on the conjugacy class of $\Gamma_P$. We will use the notion of Benjamini--Schramm convergence introduced in \cite[Sections 2-3]{7s}; this is the notion of convergence induced by the topology of weak convergence of measures on $\sub_G$. In this language Theorem \ref{main} implies that the trivial subgroup is not a limit point of the measures $\mu_P$. When $d=2$ we can prove the following result, which is much more precise than Theorem \ref{main}. 

\begin{theo} \label{BSconv-polygons}
  If $\mu$ belongs to the closure (in topology of weak convergence of measures on $\sub_G$) of the set of all $\mu_P$ for $P$ a finite-volume Coxeter polygon in $\HH^2$ then $\mu$-almost every subgroup is non-trivial and generated by reflections. 
\end{theo}

This seems likely to be true in higher dimensions as well, though our very elementary proof does not seem to immediately extend to this setting. Finally, note that any subgroup of $\isom(\HH^d)$ generated by reflections is in fact generated by the reflections in the side of a Coxeter polytope (possibly with infinitely many faces), a well-known fact\footnote{This is stated at the beginning of \cite{Vinberg_survey}, and the proof is more or less obvious. } %see Proposition \ref{wellknown} below for a proof. }. 

\subsection*{Organisation}

The proof of Theorem \ref{main} is very short and given in Section \ref{proof_thm1}. The rest of the article is dedicated to the proof of Theorem \ref{BSconv-polygons}; first we collect a few useful facts on the geometry of hyperbolic Coxeter polygons in Section \ref{facts_polygons}, and use them in Section \ref{proof_nontrivial} to prove that a Benjamini--Schramm limit of Coxeter polygons is almost surely non-trivial. Independently, we prove in Section \ref{proof_reflection} that the set of groups generated by reflections is closed in the Chabauty topology and deduce that a Benjamini--Schramm limit of Coxeter polygons is almost surely generated by reflections. 

\subsection*{Acknowledgments} Research supported by ANR grant AGDE - ANR-20-CE40-0010-01. I am grateful to David Fisher for comments on earlier versions of the paper. 

%notes
%the proof of Theorem 1 does not use coxeterness, interpretation in general??
%
%nicer proof of Theorem 5 if there was a "generalised isoperimetric inequality" for polygons in HH^2, that is if the regular ideal polygon maximalises the (volume of R-thin part)/(total volume) ratio (at least asymptotically when R\to\infty)

%%%%%%%%%%%%%%%%%%%%%%%%%%%%%%

\section{Proof of Theorem \ref{main}} \label{proof_thm1}

Fix $d \ge 2$. Let $P$ be a Coxeter polytope in $\HH^d$. We will denote by $F_i, i \in I$ the $(d-1)$-faces of $P$. 

Now let
\[
U = \{ x \in P :\: d(x, \pl P) \le 1\}
\]
and 
\[
P' = \{ x \in P :\: d(x, \pl P) \ge 1\}.
\]
We have $U \subset P_{\le 2}$ since every point of $U$ is moved by at most 2 by a reflection in a face of $P$ (as previously noted, it is clear that in fact $U = P_{\le 2}$). 

For $x \in F_i$ let $\nu_x$ be the vector normal to $F_i$ pointing inside $P$; note that $d(F_i, \exp_x(\nu_x)) = 1$ for all $x \in F_i$. Let $W_i \subset F_i$ defined by
\[
W_i = \{ x \in F_i : \exp_x(\nu_x) \in P', \, \forall \, j \not= i\, d(F_j, \exp_x(\nu_x)) > 1\}
\]
and let
\[
F_i' = \{ \exp_x(\nu_x) :\: x \in W_i \}
\]
Then the $F_i'$ are disjoint open subsets of $\pl P'$ and their complement $S = \pl P' \setminus \bigcup_{i \in I} F_i'$ is of measure 0 (with respect to the $(d-1)$-dimensional measure on $\partial P'$) as it is equal to the set of points in $\pl P'$ at distance 1 from at least two of the $F_i$. 

For $y \in F_i'$ let $\nu_y'$ be the vector orthogonal to $F_i'$ pointing outside $P'$; then the map
\begin{equation} \label{expmap}
E : [0, 1] \times \pl P' \setminus S \to \HH^d, \: (t, y) \mapsto \exp_y(t\nu_y')
\end{equation}
has its image inside $U$. Since the local geometry of the submanifolds $F_i'$ of $\HH^d$ depends only on $d$, we see that the Jacobian of $E$ is uniformly bounded away from 0 (we prove this in detail in \ref{volext} at the end of the section); let $\eps(d) > 0$ be a lower bound. Moreover, the sets $E([0, 1] \times F_i')$ are pairwise disjoint (since a point in $E([0, 1] \times F_i')$ is at distance $\le 1$ from exactly 1 face of $\pl P$, which is $F_i$, and at distance $>1$ of all others). It follows that
\[
\vol_d\left(E([0, 1] \times \pl P' \setminus S)\right) \ge \eps(d)\cdot 1 \cdot \vol_{d-1}(\pl P' \setminus S) = \eps(d) \vol_{d-1}(\pl P')
\]
so that $\vol(U) \ge \eps(d) \vol_{d-1}(\pl P')$. We finish the proof of the theorem with the following chain of inequalities: 
\begin{align*}
  \vol_d(P) &= \vol_d(U) + \vol_d(P') \\
  &\le \vol_d(U) + \vol_{d-1}(\pl P') \\
  &\le \left(1 + \eps(d)^{-1}\right) \vol_d(U) \\
  &\le C(d)\vol_d(P_{\le 2}).
\end{align*}
where the second inequality follows from the isoperimetric inequality for hyperbolic space \cite[Proposition 6.6]{gromov_measure} which implies that $\vol_d(P') \le \vol_{d-1}(\pl P')$.

%%%%%%%%%%%%%%%%%%%%%%%%%%%%%%

\subsection{Exponential map on equidistant sets} \label{volext}

Let $H$ be a geodesic hyperplane in $\HH^d$ and $H'$ a connected component of $\{x \in \HH^d :\: d(x, H) = 1\}$. Let $E : [0, 1] \times H' \to \HH^d$ be the map defined as in \eqref{expmap}. Since all hyperplanes and their equidistant sets are related by isometries, and exponential maps are equivariant with respect to those, our claim will follow if prove that the Jacobian $\det(DE(x, t))$ is uniformly bounded away from 0 for $x \in H', t \in [0, 1]$. 

The group $\isom(H)$ acts transitively on $H'$ and the map $E$ is equivariant with respect to this action. It follows that we need only to prove that $\det(DE(t, x))$ is uniformly bounded away from 0 for a fixed $x$ and $t \in [0, 1]$. This is immediate by compacity, since $DE(t, x)$ is invertible for all $t \in \RR$. 

%%%%%%%%%%%%%%%%%%%%%%%%%%%%%%%%%%%%%%%%%%%%%%%%%%%%%%%%%%%%

\section{Lemmas on Coxeter polygons} \label{facts_polygons}

We collect here some preliminary facts about Coxeter polygons in $\HH^2$, and give complete proofs for all of them; though they are likely well-known it seems more convenient to give their (short) proofs than locate sufficiently precise references for them. First we have a consequence of the collar/Margulis lemma. 

\begin{lem} \label{smalledges}
  There exists $\eta > 0$ such that if $P$ is a Coxeter polygon in $\HH^2$ then:
  \begin{enumerate}
  \item\label{rightangle} if an edge of $P$ has length $\le \eta$ then its adjacent angles are right angles; 

  \item \label{2shortedges} no two adjacent edges of $P$ have both length $\le \eta$;

\item \label{nonadjvert} any two non-consecutive vertices of $P$ are at distance at least $\eta$. 
  \end{enumerate}
\end{lem}

\begin{proof}
  Let $\Gamma = \Gamma_P$ be the discrete subgroup generated by the reflections $\sigma_e$ in the sides $e$ of $P$. Let $\delta$ be the constant given by the collar/Margulis lemma for $\HH^2$, so that for any $x \in \HH^2$ we have that
  \[
  \Gamma_x := \left\langle \gamma \in \Gamma : d(x, \gamma x) \le \delta \right\rangle
  \]
  is virtually cyclic. So if $e_2$ is an edge of $P$ with length $\le \delta/2$ and $e_1, e_3$ the adjacent edges of $P$ then the subgroup generated by $\sigma_{e_i}$ is virtually cyclic (as each of $\sigma_{e_i}$ moves any vertex of $e$ of less than $\delta$). On the other hand virtually cyclic subgroups of $\PGL_2(\RR)$ cannot contain an element of finite order other than $2$, and as this subgroup contains the rotations about the vertices of $e_2$, the angles between $e_1, e_1$ and $e_2, e_3$ must be right angles. This proves part \ref{rightangle} for any $\eta \le \delta/2$.

  For the second part we note that if there are two consecutive edges of length $\le \delta/4$ then the subgroup generated by the reflections in four consecutive edges is virtually cyclic. On the other hand it contains three reflections whose centers are not colinear, which is impossible if it is virtually cyclic. This proves part \ref{2shortedges} for any $\eta \le \delta/4$.

  To prove the last point let $v, w$ be two non-consecutive edges of $P$ and let $e'$ be the geodesic segment between $v$ and $w$. Since $P$ is a Coxeter polygon the sum of the angles of $P$ at $v$ and $w$ is at most $\pi$, so there are two edges $e_1, e_2$ of $P$ on the same side of $e'$ such that the sum of the angles between $e'$ and $e_1, e_2$ is at most $\pi/2$. Let $\beta, \gamma$ be these angles. Since the hyperbolic triangle with angles $\pi/4, \beta, \gamma$ has area at least $\pi/4$, at least one of its edges is of length $\ge d$ where $d$ is the smallest diameter of a hyperbolic disc of area $\ge\pi/4$. Either it is the side length $a$ opposite to the angle $\pi/4$, or we can assume that it is the side length $b$ opposite to the angle $\beta$ and in this case we have $\sinh(a) = \tfrac{2\sqrt 2 \sinh(b)}{\sin(\beta)} \ge 2\sqrt 2 \sinh(r)$. In any case we get that the side adjacent to the angles $\beta, \gamma$ in this triangle has length bounded below by a constant $a_0$ independent of $\beta, \gamma$. It follows that if $v, w$ are at distance less than $a_0$ then the half-lines supported by $e_1, e_2$ must intersect in $\HH^2$. It follows that the subgroup generated by the reflections in the edges of $P$ adjacent to $v, w$ contain at least 3 rotations with non-aligned center, so by Margulis lemma we must have $d(v, w) \ge \delta/2$. This proves that \ref{nonadjvert} holds for any $\eta \le \min(\delta/2, a_0)$. 
%  too complicated, just take the point groups at the 3 consecutive edges to get a contradiction with M lemma The Margulis lemma similarly implies that $P$ cannot have three consecutive edges with length $\le 3$. To get this down to two we argue as follows\footnote{It is certainly possible to be more careful to obtain an effective result. }. If $P$ has two consecutive edges $a, b$ both of length $\le\eta$ let $c_1, c_2$ be the other edges adjacent to $a, b$ respectively. By the first part the angles between $a, c_1$ and $b, c_2$ are both $\pi/2$. By taking $\eta$ small enough we can make sure that the half-lines prolonging $c_1, c_2$ away from $a, b$ intersect. It follows that $P$ must be contained in the quadrilateral formed by $c_1, c_2$ and the prolongations of $a, b$. Let $\gamma$ be the fourth angle ; by \cite[Theorem 3.5.10]{Ratcliffe} we have that $\cos(\gamma) \le \sinh(\eta)^2$, and it follows that by taking $\eta$ small enough we can make $\gamma$ arbitrarily close to $\pi/2$, hence the area of the quadrilateral arbitrarily small. But since $P$ is a Coxeter polygon its area is at least $\pi/42$ by a result of Siegel \cite{Siegel} so we obtain a contradiction. 
\end{proof}

Next we will need the following lemma on the area of hyperbolic triangles.

\begin{lem} \label{nosmallarea}
  For any $\ell_0$ there exists a constant $A > 0$ such that any hyperbolic triangle with edges of length at least $\ell_0$ and angles at most $\pi/2$ has area at least $A$. 
\end{lem}

\begin{proof}
  Let $a, b, c$ be the edge lengths of such a triangle $T$ and define
  \[
  s = \frac{a+b+c}2,\, s_a = s-a,\, s_b=s-b,\, s_c=s-c. 
  \]
  The hyperbolic Heron formula \cite[Theorem 1.1(i)]{Mednykh_brahmagupta} states that
  \[
  \sin(\area(T)) = \frac{\sqrt{\sinh(s)\sinh(s_a)\sinh(s_b)\sinh(s_c)}}{4\cosh(a/2)\cosh(b/2)\cosh(c/2)}.
  \]
  Using the hyperbolic cosine law it follows from our hypotheses on $T$ that there exists a constant $\ell_1 > 0$ (independent of $a, b, c$) such that $s, s_a, s_b, a_c \ge \ell_1$. It follows that we have $\sinh(s_a) \gg e^{s_a}$ and similarly for the other terms. As $s_a+s_b+s_c = 2s$ we get that
  \[
  \sinh(s)\sinh(s_a)\sinh(s_b)\sinh(s_c) \gg e^{3s}.
  \]
  Similarly,
  \[
  \cosh(a/2)\cosh(b/2)\cosh(c/2) \ll e^s
  \]
  so in the end
  \[
  \area(T) \ge \sin(\area(T)) \gg e^{s/2}
  \]
  which finishes the proof. 
\end{proof}

Finally we will use the following lemma. 

\begin{lem} \label{remove_small_edges}
  Let $\eps > 0$. There exists a constant $\eta' > 0$ such that for any Coxeter polygon $P$ in $\HH^2$ and any $R$ there exists a polygon $P'$ such that
  \begin{enumerate}
  \item \label{nosmalledges} $P'$ has no edge length smaller than $\eta'$ ;

  \item \label{volumechange} $\tfrac{\area \left(P_{\le R}'\right) }{\area P'} \le 2 \tfrac{\area \left(P_{\le R}\right) }{\area P}$.
  \end{enumerate}
\end{lem}

\begin{proof}
  We construct $P'$ as follows : let $v_1, \ldots, v_m$ be a cyclic ordering of the vertices of $P$. Let $\eta$ be the constant from Lemma \ref{smalledges}, and $0 < \alpha < 1/2$ (to be determined later). For each pair $(v_i, v_{i+1})$ of adjacent vertices such that $d(v_i, v_{i+1}) \le \alpha \cdot \eta$ we remove the vertex $v_{i+1}$ from $P$, that is, if $i_1, \ldots, i_k$ are those indices such that $d(v_{i_j}, v_{i_j+1}) \le \alpha \cdot \eta$ we take $P'$ to be the polygon spanned by the $v_i, i \not\in \{i_1, \ldots, i_k\}$. Since $d(v_{i_j+1}, v_{i_j+2}) > \eta$ by Lemma \ref{smalledges}, it follows from the triangle inequality that $d(v_{i_j}, v_{i_j+2}) > (1-\alpha)\eta > \alpha\eta$ and so $P'$ satifies condition \ref{nosmalledges} for any $\eta' \le \alpha\eta$. 

\medskip
  
  We now prove that $P'$ satisfies condition \ref{volumechange} for sufficiently small $\alpha$. First we estimate the area of each removed triangle. To do this let $i$ such that $d(v_i, v_{i+1}) < \alpha\eta$; we want to estimate the area of the triangle $T_i$ spanned by $v_i, v_{i+1}, v_{i+2}$. Let $\gamma$ be its angle at $v_i$, $c = d(v_{i+1}, v_{i+2})$, $a = d(v_i, v_{i+2})$; by the hyperbolic sine law we have that $\sinh(a) = \tfrac{\sinh(c)}{\sin(\gamma)}$. We compute that
  \begin{align*}
    \sin(\gamma) &\ge \frac{\sinh(a-\eta)}{\sinh(a)} \\
    &\ge \frac{\sinh(a) - \eta\cosh(a)}{\sinh(a)} \ge 1 - u\cdot\alpha
  \end{align*}
  where $u > 0$ on the last line depends only on $\eta$. It follows that $\gamma \ge \tfrac \pi 2 - u' \cdot \alpha$, and finally that 
  \begin{equation}\label{smallareatrig}
    \area(T_i) \le u' \cdot \alpha 
  \end{equation}
  for some $u'$ independent of $P, \alpha$. 

  Now the triangle $S_i$ spanned by $v_{i-1}, v_i, v_{i+2}$ has all its edge length at least $\eta$, by Lemma \ref{smalledges}. Its angles are at most $\pi/2$ since it is inscribed in the Coxeter polygon $P$. So by Lemma \ref{nosmallarea} we have that $\area(S_i) \ge A$. By \eqref{smallareatrig} this implies that $\area(T_i) = O(\area(S_i)\alpha)$ uniformly in $i, P$, and as $\area(P')  = \area(P) - \sum_{l=1}^k \area(T_{i_l})$ it follows that
  \[
  \area(P') = (1-O(\alpha))\area(P).
  \]
  On the other hand, we have that $P_{\le R}'$ is contained in the $\alpha\eta$-neighbourhood of $P_{\le R}$ so that
  \[
  \area((P')_{\le R} \le (1+O(\alpha))\area(P_{\le R}).
  \]
  From these two inequalities it follows that
  \[
  \frac{\area((P')_{\le R})}{\area(P')} \le (1+O(\alpha)) \frac{\area(P_{\le R})}{\area(P)}
  \]
  which finishes the proof by taking $\eta' = \alpha\eta$ for $\alpha$ small enough (independently of $P, R$).   

\end{proof}

%%%%%%%%%%%%%%%%%%%%%%%%%%%%%%%%%%%%%%%%%%%%%%%%%%%%%%%%%%%%

\section{Nontriviality of BS-limits of polygons} \label{proof_nontrivial}

In this section we give the proof of the first part of Theorem \ref{BSconv-polygons}, that a Benjamini--Schramm limit of a sequence of Coxeter polygons is almost surely nontrivial. We will prove that for any sequence $P_n$ of Coxeter polygons in $\HH^2$ we have
\begin{equation} \label{geom_BSlim_polygon}
  \lim_{R \to +\infty} \lim_{n \to +\infty} \frac{\area \left((P_n)_{\ge R}\right)} {\vol P_n} = 0. 
\end{equation}
from which the first statement follows immediately. 

\medskip

Let $P_n'$ be the polygons obtained from the $P_n$ by applying Lemma \ref{remove_small_edges}; it follows from the condition \eqref{nosmalledges} that it suffices to prove \eqref{geom_BSlim_polygon} for the $P_n'$. Note that in any triangulation\footnote{By this we mean a decomposition of a polygon into triangles whose vertices are vertices of the polygon. } of $P_n'$ all angles are smaller than $\pi/2$ (they are smaller than those of the $P_n$ and the latter are of the form $\pi/k$, $k \ge 2$), and by using point \eqref{nonadjvert} of Lemma \ref{smalledges} in addition we see that all edges in the triangulation have length at least $\eta$, so by Lemma \ref{nosmallarea} we have a uniform lower bound for all areas of triangles occuring in any triangulation of any $P_n'$. 

We triangulate $P_n'$ as follows : we choose a vertex, and as long as possible add and edge between the current vertex and the second-to-next one (clockwise). See Figure \ref{triangulation} for an illustration. 

\begin{figure}
    \centering
    \begin{subfigure}[b]{0.4\textwidth}
        \includegraphics[width=\textwidth]{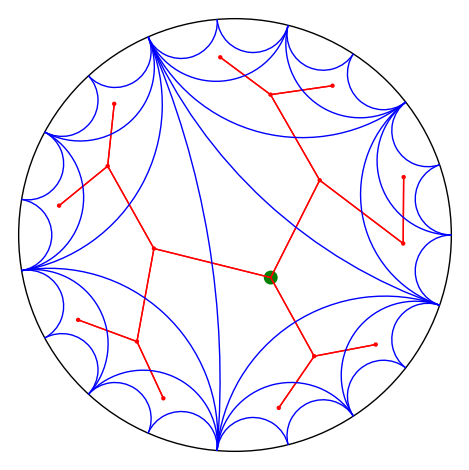}
        \caption{$n=19$}
    \end{subfigure}
    \hspace{1cm}
    \begin{subfigure}[b]{0.4\textwidth}
        \includegraphics[width=\textwidth]{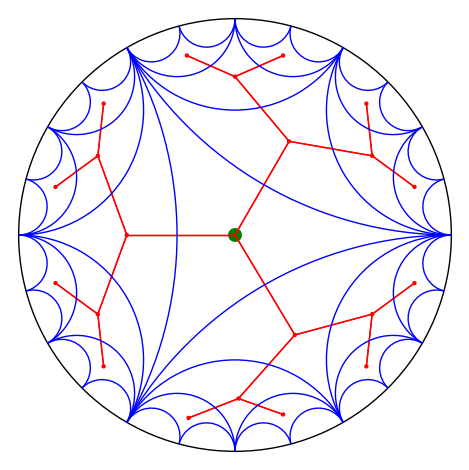}
        \caption{$n=24$}
    \end{subfigure}
    \caption{Triangulations of ideal $n$-gons and their dual tree}
    \label{triangulation}
\end{figure}

The tree $T_n$ dual to this triangulation, rooted at the last triangle, has radius at most $\log_2(n) + 1$ by \cite[Proposition 2.1]{Korman}. Moreover the distance of the closest leaf to this triangle is at least $\log_2(n) - 1$ as well\footnote{This can easily be seen by observing that these functions are monotonic, and for a $2^k$-gon thay are equal to $\log_2(n) \pm 1$ (it would be more natural in this case to center at the edge between the last two triangles). }.

%We denote $k = \lfloor \log_2(n) \rfloor$.
Let $S > 0$. Let $(T_n)_{\le S}$ be the set of vertices in $T_n$ which are at distance at most $S$ from $\partial T_n$\footnote{For us the boundary of a finite tree is its set of leaves. }. Then $T_n \setminus (T_n)_{\le S}$ is contained in the ball of radius $\log_2(n) + 1 - S$ since every leaf is at distance at most $\log_2(n) + 1$ from the root. It follows that $|T_n \setminus (T_n)_{\le S}| \le 2^{\log_2(n) + 1 - S} = O(2^{-S}|T_n|$ since $|T_n| = n$. So we have 
\begin{equation} \label{boundary_trig}
  |(T_n)_{\le S}| \ge (1 - O(2^{-S})) |T_n| .
\end{equation}
On the other hand, if $x \in P_n'$ lies in a triangle corresponding to a triangle $t \in (T_n)_{\le S}$ then we may construct a path $c$ as follows: let $t=t_0$; choose a side of $t_0$ closest to $x$ such that the corresponding edge of $T_n$ points away from the root, and let $c_0$ be the geodesic from $x$ to this side. If the side lies on the boundary let $c=c_0$; otherwise let $t_1$ be the triangle on the other side, let $x_1$ be the foot of $c_0$ on the side $t_0\cap t_1$ and iterate the construction until the boundary is reached, say in $l$ steps, and let $c$ be the concatenation of $c_0, \ldots, c_l$ where $c_i$ is the path obtained at the $(i+1)$th step. As $t \in (T_n)_{\le S}$ and the edges $(t_i, t_{i+1})$ point away from the root we must have that $l \le S+1$. Moreover, the length of each $c_i$ is at most $\log(3)$ since $\HH^2$ is $\log(3)/2$-hyperbolic so at each step at least one of the two edges pointing away from the root is at distance $\le \log(3/2)$ from the foot of $c_i$. It follows that $c$ has length at most $(S+1)\log(3)$ and we conclude that $x$ is at distance at most $\log(3) \cdot (S+1)$ from the boundary. From this and \eqref{boundary_trig} we deduce that at least $(1 - O(2^{-S})) |T_n|$ triangles of $T_n$ lie entirely in $(P_n')_{\le R}$, for $S = R/\log(3)$. If $A$ is a lower bound for the area of triangles in $T_n$ (which is independent of $n$ by the remarks above) we thus have that
%, \ldots, t_l \in \partial T_n$ be the geodesic in $T_n$ from $t$ to a leaf. Let $c_0$ be the geodesic from $x$ to the side separating $t_0$ from $t_1$, and afterwards define $c_{i+1}$ to be the geodesic from the end of $c_i$ to the side separating $t_i$ from $t_{i+1}$, and finally $c_l$ as a path from the end of $c_{l-1}$ to the closest point on another side of $t_l$ (which lies in the boundary of $P_n'$). Finally let $c$ be the concatenation of $c_0, \ldots, c_l$; this is a path from $x$ to $\partial P_n'$. Moreover, since $\HH^2$ is $\log(3)/2$ hyperbolic the length of each $c_i$ is at most $\log(3)$, so we conclude that $x$ is at distance at most $\log(3) \cdot (S+1)$ from the boundary. From this and \eqref{boundary_trig} we deduce that at least $(1 - O(2^{-S})) |T_n|$ triangles of $T_n$ lie entirely in $(P_n')_{\le R}$, for $S = R/\log(3)$. If $A$ is a lower bound for the area of triangles in $T_n$ (which is independent of $n$ by the remarks above) we thus have that
\[
\area(P_n') \ge (1 - O(2^{-R}) A.
\]
On the other hand we have that
\[
\area(P_n' \setminus (P_n')_{\le R}) \le 2^{\log_2(n) + 1 - S} \pi
\]
so that 
\[
\area((P_n')_{\le R}) \ge (1 - O(2^{-R}) \area(P_n').
\]
% need to be a bit more precise here: first area(P_n') \ge A |T_n|, then $area(P_n' - (P_n'_){\le R} \le 3.2^{k-S} \pi$ -> OK
from which \eqref{geom_BSlim_polygon} follows immediately. 

%%%%%%%%%%%%%%%%%%%%%%%%%%%%%%%%%%%%%%%%%%%%%%%%%%%%%%%%%%%%

\section{Chabauty limits of Coxeter polygons} \label{proof_reflection}

In this section we prove the following result, which immediately implies the second part of Theorem \ref{BSconv-polygons} since limits of discrete invariant random subgroups are themselves supported on discrete subgroups (as follows from \cite[Proposition 2.2, Theorem 2.9]{7s}). 

\begin{prop} \label{coxeter_closed}
  The set of discrete groups generated by reflections is closed in the Chabauty space of discrete subgroups of $\PO(2, 1)$. 
\end{prop}

This will follow from the next lemma.

\begin{lem} \label{gens_minlength}
  Let $P$ be a Coxeter polygon in $\HH^2$, $S$ the set of reflections in its faces and $\Gamma = \Gamma_P = \langle S \rangle$. If $w \in \Gamma$ and $T$ is the subset of $S$ containing all elements occuring in a minimal expression for $w$ as a word in the elements of $S$ then $d(x, wx) \ge d(x, sx)$ for all $s \in T$. 
\end{lem}

\begin{proof}
  Assume that $s \in S$ occurs in a minimal expression for $w$; then $x$ and $wx$ are separated by an hyperplane which is $\Gamma$-equivalent to the hyperplane $W_s$ supporting the side of $P$ corresponding to $s$. Thus we need only show the following statement: the smallest distance between two points in the orbit $\Gamma \cdot x$ separated by $W_s$ is realised by $(x, sx)$.

  In turn this is implied by the statement that the closest point to $x$ on a $\Gamma$-translate of $W_s$ is its projection on $W_s$. If this were not the case there would be a billiard trajectory in $P$ starting at $x$ and ending on $W_s$ shorter than the segment from $x$ orthogonal to $W_s$. This is impossible: indeed, there is no trajectory from $x$ to $W_s$ at all which is shorter than this segment. This finishes the proof. 
\end{proof}

\begin{proof}[Proof of Proposition \ref{coxeter_closed}]
  Let $H$ be a discrete Chabauty limit of a sequence $\Gamma_n$ of Coxeter groups in $\HH^2$. Since $H$ is discrete, it follows from the Kazhdan--Margulis theorem there exists $\eps > 0$ and a point $o$ which is contained in the $\eps$-thick part of $\Gamma_n \bs \HH^2$ for every $n$. Let $P_n$ be the Coxeter polygon of $\Gamma_n$ containing $o$.

  Fix a $g \in H$: we want to prove that $g$ is a product of reflections in $H$. We know that $g$ is a limit of a sequence $g_n \in \Gamma_n$. The distances $d(o, g_n o)$ are uniformly bounded, say $d(o, g_n o) \le R$ for all $n$, and it follows that the word length of $g_n$ in the reflections in the sides of $P_n$ must be bounded, say by some $l \in \mathbb N$ (since there are uniformly finitely many points in the orbit $\Gamma_n \cdot o$ at distance at most $R$ from $o$ by a packing argument). By Lemma \ref{gens_minlength} it follows that we have $g_n = s_{i_1, n} \cdots s_{i_{l_n}, n}$ where $s_{i, n}$ are reflections in the sides of $P_n$ such that $d(o, s_{i, n} o) \le R$ and $l_n \le l$ for all $n$. We can thus pass to a subsequence and assume that the sequences $s_{i, n}$, $1 \le i \le l$ are convergent. The limits are reflections belonging to $H$ and $g$ can be written as a product of them. 
\end{proof}

\bibliographystyle{alpha}
\bibliography{bib}

\newcommand{\etalchar}[1]{$^{#1}$}
\begin{thebibliography}{KLM{\etalchar{+}}18}

\bibitem[ABB{\etalchar{+}}17]{7s}
Miklos Abert, Nicolas Bergeron, Ian Biringer, Tsachik Gelander, Nikolay
  Nikolov, Jean Raimbault, and Iddo Samet.
\newblock On the growth of {{\(L^2\)}}-invariants for sequences of lattices in
  {Lie} groups.
\newblock {\em Ann. Math. (2)}, 185(3):711--790, 2017.

\bibitem[ABSW08]{ABSW}
Ian Agol, Mikhail Belolipetsky, Peter Storm, and Kevin Whyte.
\newblock Finiteness of arithmetic hyperbolic reflection groups.
\newblock {\em Groups Geom. Dyn.}, 2(4):481--498, 2008.

\bibitem[Bel16]{Belolipetsky_survey}
Mikhail Belolipetsky.
\newblock Arithmetic hyperbolic reflection groups.
\newblock {\em Bull. Am. Math. Soc., New Ser.}, 53(3):437--475, 2016.

\bibitem[Dav08]{Davis_book}
Michael~W. Davis.
\newblock {\em The geometry and topology of {Coxeter} groups.}, volume~32 of
  {\em Lond. Math. Soc. Monogr. Ser.}
\newblock Princeton, NJ: Princeton University Press, 2008.

\bibitem[FH22]{fisher_hurtado}
David Fisher and Sebastian Hurtado.
\newblock A new proof of finiteness of maximal arithmetic reflection groups,
  2022.

\bibitem[FHR22]{aritmargulis}
Mikołaj Fr{\c a}czyk, Sebastian Hurtado, and Jean Raimbault.
\newblock Homotopy type and homology versus volume for arithmetic locally
  symmetric spaces, 2022.

\bibitem[Gro07]{gromov_measure}
Misha Gromov.
\newblock {\em Metric structures for {Riemannian} and non-{Riemannian} spaces.
  {Transl}. from the {French} by {Sean} {Michael} {Bates}. {With} appendices by
  {M}. {Katz}, {P}. {Pansu}, and {S}. {Semmes}. {Edited} by {J}. {LaFontaine}
  and {P}. {Pansu}}.
\newblock Mod. Birkh{\"a}user Classics. Basel: Birkh{\"a}user, 3rd printing
  edition, 2007.

\bibitem[Kho86]{Khovanskii}
A.~G. Khovanskij.
\newblock Hyperplane sections of polyhedra, toroidal manifolds, and discrete
  groups in {Lobachevskij} space.
\newblock {\em Funct. Anal. Appl.}, 20:41--50, 1986.

\bibitem[KLM{\etalchar{+}}18]{Korman}
Matias Korman, Stefan Langerman, Wolfgang Mulzer, Alexander Pilz, Maria
  Saumell, and Birgit Vogtenhuber.
\newblock The dual diameter of triangulations.
\newblock {\em Comput. Geom.}, 68:243--252, 2018.

\bibitem[Lak12]{lakeland}
Grant~S. Lakeland.
\newblock Dirichlet-{Ford} domains and arithmetic reflection groups.
\newblock {\em Pac. J. Math.}, 255(2):417--437, 2012.

\bibitem[Med12]{Mednykh_brahmagupta}
A.~D. Mednykh.
\newblock Brahmagupta formula for cyclic quadrilaterials in the hyperbolic
  plane.
\newblock {\em Sib. {\`E}lektron. Mat. Izv.}, 9:247--255, 2012.

\bibitem[Nik07]{nikulin}
V.~V. Nikulin.
\newblock Finiteness of the number of arithmetic groups generated by
  reflections in {Lobachevsky} spaces.
\newblock {\em Izv. Math.}, 71(1):53--56, 2007.

\bibitem[Vin81]{Vinberg_35}
{\`E}.~B. Vinberg.
\newblock Absence of crystallographic reflection groups in
  {Lobachevski{\u{\i}}} spaces of large dimension.
\newblock {\em Funkts. Anal. Prilozh.}, 15(2):67--68, 1981.

\bibitem[Vin85]{Vinberg_survey}
{\`E}.~B. Vinberg.
\newblock Hyperbolic reflection groups.
\newblock {\em Russ. Math. Surv.}, 40(1):31--75, 1985.

\end{thebibliography}

\end{document}